\documentclass[12pt, a4paper]{amsart}

\usepackage[T1]{fontenc}
\usepackage[utf8]{inputenc}
\usepackage[margin=2.0cm]{geometry}
\usepackage{hyperref}
\usepackage{enumerate}
\usepackage{amssymb}
\usepackage{amsmath}
\usepackage{latexsym}
\usepackage{amsthm}
\usepackage{amsmath}
\usepackage[english]{babel}
\usepackage{tikz}
\usepackage{lipsum}

\newtheorem{theorem}{Theorem}[]
\newtheorem{lemma}[theorem]{Lemma}
\newtheorem{proposition}[theorem]{Proposition}
\newtheorem{cor}[theorem]{Corollary}
\theoremstyle{definition}

\theoremstyle{remark}

\numberwithin{equation}{section}

\newcommand{\mysetminusD}{\hbox{\tikz{\draw[line width=0.6pt,line cap=round] (3pt,0) -- (0,6pt);}}}
\newcommand{\mysetminusT}{\mysetminusD}
\newcommand{\mysetminusS}{\hbox{\tikz{\draw[line width=0.45pt,line cap=round] (2pt,0) -- (0,4pt);}}}
\newcommand{\mysetminusSS}{\hbox{\tikz{\draw[line width=0.4pt,line cap=round] (1.5pt,0) -- (0,3pt);}}}
\newcommand{\mysetminus}{\mathbin{\mathchoice{\mysetminusD}{\mysetminusT}{\mysetminusS}{\mysetminusSS}}}

\newcommand{\Q}{\mathbb{Q}}

\title{On Widmer's criteria for the Northcott property}

\author{Sara Checcoli}
\address[Sara Checcoli]{Univ. Grenoble Alpes, CNRS, IF, 38000 Grenoble, France}
\email{sara.checcoli@univ-grenoble-alpes.fr}

\author{Arno Fehm}
\address[Arno Fehm]{Technische Universit\"at Dresden, Fakult\"at Mathematik, Institut f\"ur Algebra, 01062 Dresden, Germany} 
\email{arno.fehm@tu-dresden.de}

\begin{document}

\maketitle

\begin{abstract}
    Recently, Widmer introduced a new sufficient criterion for the Northcott property on the finiteness of elements of bounded height
    in infinite algebraic extensions of number fields. 
    We provide a simplification of Widmer's criterion when the extension is abelian,
    and use this to exhibit fields with the Northcott property that do not satisfy Widmer's new criterion.
    We also show how the construction of pseudo algebraically closed fields with the Northcott property, carried out in previous work using the first version of Widmer's criterion, can be simplified.
\end{abstract}

\section{Introduction}
\maketitle

\noindent
Let $h$ denote the absolute logarithmic Weil height on 
the field $\overline{\mathbb{Q}}$ of algebraic numbers. 
A celebrated theorem of Northcott \cite{Northcott} gives that 
for every $B>0$ and $D>0$ there are only finitely many $\alpha\in\overline{\mathbb{Q}}$ with $h(\alpha)<B$ and $[\mathbb{Q}(\alpha):\mathbb{Q}]<D$,
cf.~\cite[Theorem 1.6.8]{BG}.
Accordingly, one says that
a subfield $L$ of $\overline{\mathbb{Q}}$ has the \emph{Northcott property} if
for every $B>0$ the set of $\alpha\in L$ with $h(\alpha)<B$ is finite.
This property was initially introduced and explored by Bombieri and Zannier in \cite{BZ}. While its validity for number fields
follows directly from Northcott's theorem, 
establishing it for infinite extensions of the rationals remains a challenging task. 
A comprehensive survey on the Northcott property, including its connections to algebraic dynamics and decidability in logic, can be found in \cite{Wid23}.

A first example of an infinite algebraic extension of $\mathbb{Q}$ with the Northcott property was given by Bombieri and Zannier \cite[Theorem 1]{BZ}, who showed that for any positive integer $d$ and any number field $K$, 
the compositum $K^{(d)}_{\rm ab}$ of all abelian extensions of $K$ of degree at most $d$ has the Northcott property.
A general sufficient criterion was subsequently introduced by Widmer in \cite[Theorem 3]{Wid11}: 
He demonstrated that the Northcott property holds for the union of an infinite tower of number fields, whenever the relative discriminants have \emph{sufficiently rapid growth} at each non-trivial step within the tower. This enabled him to provide new examples of fields with the Northcott property that were not accounted for in the result of Bombieri and Zannier.

Already Bombieri and Zannier in \cite{BZ}
had given a general sufficient condition
for the Northcott property for algebraic extensions of $\mathbb{Q}$, based on the splitting of primes,
but they speculated that their criterion might be satisfied only for number fields.
In \cite[Theorem~1.5]{CF}, 
the authors of the present paper
constructed examples of infinite algebraic extensions of $\mathbb{Q}$
not contained in $K^{(d)}_{\rm ab}$
to which this criterion of Bombieri and Zannier can be applied, while Widmer's criterion cannot.

Very recently, Widmer in \cite{Wid23} introduced an improved version of his criterion and proved that it not only implies the result from \cite{Wid11} but also applies to $K^{(d)}_{\rm ab}$. To state Widmer's result we need to introduce some notation.
For $M/K$ an extension of number fields, let $D_{M/K}$ denote the relative discriminant 
and $N_{M/K}$ the relative norm,
and 
\[
 \gamma(M/K)=\sup_{F/K \textrm{ finite}} N_{F/\Q}(D_{MF/F})^{\frac{1}{[MF:\Q][MF:F]}},
\]
where $MF$ is the compositum in $\overline{\mathbb{Q}}$.
Widmer's new criterion is the following:

\begin{theorem}[{\cite[Theorem 8]{Wid23}}]\label{new-crit-Wid}
Let $K$ be a number field and let $L$ be an infinite algebraic extension of $K$. 
Suppose that 
\begin{equation}\label{eqn:liminf}
\liminf_{K\subseteq M\subset L}\gamma({M/K})=\infty
\end{equation}
where $M$ runs over all number fields in $L$ containing $K$. Then $L$ has the Northcott property.
\end{theorem}

The first aim of the present note is to show how to slightly simplify Widmer's criterion in certain situations. 
More precisely, in Section \ref{sec-simp-Wid} Corollary \ref{simply-Wid},  
we prove that if $M/K$ is a Galois extension, to compute $\gamma({M/K})$ it is enough to consider the supremum only over finite extensions $F$ of $K$ that are \emph{contained in} $M$, rather than over \emph{all} finite extensions of $K$.
This can simplify the verification of the failure of assumption (\ref{eqn:liminf}) in Theorem \ref{new-crit-Wid} in certain cases, for instance when $L/K$ is abelian (as in several results in \cite{BZ} and \cite{CF}). 

In Section~\ref{sec-field-N-not-Wid}
we then use Corollary \ref{simply-Wid} to show that the construction in the proof of \cite[Theorem 1.5]{CF} yields infinite Galois extensions that resist even the new criterion of Widmer, confirming the criterion from \cite{BZ} as the only alternative method for establishing the Northcott property that is currently known:

\begin{theorem}\label{main-th}
There exists an infinite Galois extension $L/\Q$ with the Northcott property
that does not satisfy condition $(\ref{eqn:liminf})$ of Theorem \ref{new-crit-Wid} for any number field $K\subseteq L$.
\end{theorem}

Finally, we revisit the connection between the Northcott property and PAC (pseudo algebraically closed) fields, i.e.,  fields over which every geometrically integral variety has a rational point. 
Motivated by a question of Amoroso, David and Zannier \cite{ADZ}, 
the second author of this paper 
constructed in \cite{Feh} PAC fields with the Northcott property using the criterion from \cite{Wid11}. 
In Section \ref{sec-PAC}, we demonstrate how this construction can be made to work without Widmer's criterion,
thereby simplifying it significantly. 
This result stands independently, and its proof is not reliant on Sections \ref{sec-simp-Wid} and \ref{sec-field-N-not-Wid}.

\section{Simplifying the criterion for Galois extensions}\label{sec-simp-Wid}

\noindent
Let $M/K$ be a finite extension of number fields.
Recall from the introduction that we denote by $D_{M/K}$ its relative discriminant and by $N_{M/K}$ the relative norm. 
Morever, $\mathcal{O}_M$ denotes the ring of integers 
and $\Delta_M$ the absolute discriminant of $M$.
Given another number field $F$ let 
\[
 \gamma_{M}(F)=N_{F/\Q}(D_{MF/F})^{\frac{1}{[MF:\Q][MF:F]}}.
\]
We define the quantity
\[
 \gamma(M/K)=\sup_{F/K \textrm{ finite}} \gamma_{M}(F),
\]
as in \cite[formula (1.4)]{Wid23} and the new quantity
\[
 \gamma'({M/K})=\sup_{\substack{F/K \textrm{ finite}\\ F\subseteq M}} \gamma_M(F).
\]
We will need the following result that can be found in \cite{Toy} and is recalled here for clarity:

\begin{lemma}\label{toy}
Let $K$ be a number field and let $L/K$ and $L'/K$ be two finite extensions. Then $D_{L L'/K}$ divides the product \[D_{L/K}^{[L L':L]}\cdot D_{L'/K}^{[L L':L']}.\]
\end{lemma}

\begin{proof}
We denote by $\delta_{L/K}$ the different ideal of $L/K$.
For any finite extension $M/K$,
$\delta_{M/K}$ is generated by all the $f_a'(a)$ where $a\in \mathcal{O}_M$ generates $M/K$ and $f_a(x)\in \mathcal{O}_K[x]$ is its minimal polynomial,
see \cite[Thm.~III.2.5]{Neu}.
So since $L=K(a)$ implies $LL'=L'(a)$ and the minimal polynomial of $a$ over $L'$ divides the minimal polynomial of $a$ over $K$,
we obtain that $\delta_{L/K}\subseteq\delta_{LL'/L'}$.
Thus, $\delta_{LL'/L'}$ divides $\delta_{L/K}\mathcal{O}_{LL'}$.
So since $\delta_{L L'/K}=\delta_{L L'/L'}\delta_{L'/K}$ 
by the multiplicativity of differents
(\cite[Prop.~III.2.2(i)]{Neu}),
$\delta_{L L'/K}$
divides $\delta_{L/K}\delta_{L'/K}\mathcal{O}_{LL'}$.
As the norm of the different is the discriminant 
(\cite[Thm.~III.2.9]{Neu}),
we obtain that
$D_{L L'/K}=N_{L L'/K}(\delta_{L L'/K})$ divides 
\[
 N_{L L'/K}(\delta_{L/K}\delta_{L'/K}\mathcal{O}_{LL'})=D_{L/K}^{[L L':L]}D_{L'/K}^{[L L':L']},
\]
cf.~\cite[Prop.~III.1.6(i),(ii)]{Neu}.
\end{proof}

\begin{proposition}\label{simply-Wid-prop}
Let $M$ and $F$ be two number fields which are linearly disjoint over their intersection $M\cap F$. Then $\gamma_M(F)\leq \gamma_M(M\cap F)$.
\end{proposition}

\begin{proof}
Let $E=M\cap F$.
We will use several times that for number fields $K\subseteq L$,
\begin{equation}\label{eqn:disc_tower}
 |\Delta_{L}| = |N_{K/\mathbb{Q}}(D_{L/K})|\cdot|\Delta_{K}|^{[L:K]},    
\end{equation}
cf.~\cite[Cor.~III.2.10]{Neu}.
In particular, 
\begin{eqnarray}
\label{eqn:gamma_M_F}\gamma_M(F)&=&\frac{|\Delta_{MF}|^{\frac{1}{[MF:\Q][MF:F]}}}{|\Delta_F|^{\frac{1}{[MF:\Q]}}},\\
\label{eqn:gamma_M_E}\gamma_M(E)&=&\frac{|\Delta_{M}|^{\frac{1}{[M:\Q][M:E]}}}{|\Delta_{E}|^{\frac{1}{[M:\Q]}}}.
\end{eqnarray}
By Lemma \ref{toy}, $D_{MF/E}$ divides
$D_{M/E}^{[MF:M]}\cdot D_{F/E}^{[MF:F]}$,
so by the multiplicativity of norms we get
\begin{eqnarray*}
|N_{E/\Q}(D_{MF/E})|&\leq& |N_{E/\Q}(D_{M/E})|^{[MF:M]}\cdot|N_{E/\Q}(D_{F/E})|^{[MF:F]}.
\end{eqnarray*}
Thus 
\begin{eqnarray*}
|\Delta_{MF}|&\stackrel{(\ref{eqn:disc_tower})}{=}&|N_{E/\Q}(D_{MF/E})|\cdot|\Delta_{E}|^{[MF:E]}\\
&\leq&  |N_{E/\Q}(D_{M/E})|^{[MF:M]}\cdot|N_{E/\Q}(D_{F/E})|^{[MF:F]}\cdot|\Delta_{E}|^{[MF:E]}\\
&=&\left(|N_{E/\Q}(D_{M/E})|\cdot|\Delta_{E}|^{[M:E]}\right)^{[MF:M]}\cdot|N_{E/\Q}(D_{F/E})|^{[MF: F]}\\
&\stackrel{(\ref{eqn:disc_tower})}{=}&|\Delta_M|^{[MF: M]}\cdot|N_{E/\Q}(D_{F/E})|^{[MF: F]}\\
&\stackrel{(\ref{eqn:disc_tower})}{=}&\frac{|\Delta_M|^{[MF:M]} \cdot|\Delta_F|^{[MF:F]}}{|\Delta_{E}|^{[MF:E]}}.
\end{eqnarray*}
Therefore 
\begin{eqnarray*}
\gamma_{M}(F)&\stackrel{(\ref{eqn:gamma_M_F})}{=}&\frac{|\Delta_{MF}|^{\frac{1}{[MF:\Q][MF:F]}}}{|\Delta_F|^{\frac{1}{[MF:\Q]}}}\\
&\leq& \left(\frac{|\Delta_M|^{[MF:M]}\cdot |\Delta_F|^{[MF:F]}}{|\Delta_{E}|^{[MF:E]}}\right)^{\frac{1}{[MF:\Q][MF:F]}} \cdot\frac{1}{|\Delta_F|^{\frac{1}{[MF:\Q]}}}\\
&\stackrel{(*)}{=}&\frac{|\Delta_M|^\frac{1}{[M:\Q][M:E]}}{|\Delta_{E}|^{\frac{1}{[E:\Q][M:E]}}}\\
&\stackrel{(\ref{eqn:gamma_M_E})}{=}&\gamma_M(E)
\end{eqnarray*}
where in $(*)$ we used that $[MF:F]=[M:E]$ since $M$ and $F$ are linearly disjoint over $E$.
\end{proof}

\begin{cor}\label{simply-Wid} 
Let $K$ be a number field and $M/K$ a finite Galois extension.
Then $\gamma'({M/K})=\gamma({M/K})$.
\end{cor}

\begin{proof}
As $M/K$ is Galois, given a finite extension $F/K$,  $M$ and $F$ are linearly disjoint over $M\cap F$ (see for instance \cite[Cor.~2.5.2]{FJ}), so by Proposition \ref{simply-Wid-prop} we have $\gamma_M(F)\leq \gamma_M(M\cap F)$.
Thus
\begin{equation*}
 \gamma(M/K) = \sup_{{F/K \textrm{ finite}}}\gamma_M(F)  = 
  \sup_{\substack{F/K \textrm{ finite}\\ F\subseteq M}}\gamma_M(F) = \gamma'(M/K). \qedhere   
\end{equation*}
\end{proof}

\section{An example to which the criterion does not apply: proof of Theorem \ref{main-th}}\label{sec-field-N-not-Wid}

\noindent
We start with a result providing 
an upper bound for the quantity $\gamma'$ in certain situations.
We then apply this to the construction from \cite[Theorem 1.5]{CF}.

\begin{proposition}\label{main-details}
Let  $(p_i)_{i\in\mathbb{N}}$ be a sequence of distinct prime numbers and let $(L_i)_{i\in\mathbb{N}}$ be a sequence of number fields satisfying the following properties:
\begin{enumerate}[(i)]
\item\label{itGal} $L_i/\Q$ is a cyclic Galois extension of degree $p_i$, for each $i$;
\item\label{itdisc} the discriminants $\Delta_{L_i}$ of the number fields $L_i$ are pairwise coprime;
\end{enumerate}
Then for all number fields $K\subseteq M$ contained in the compositum $L$ of the $L_i$ we have
\[
 \gamma'({M}/K)\leq \sup_{i\in\mathbb{N}}|\Delta_{L_i}|^{{1/p_i^2}}.
\]
\end{proposition}

\begin{proof}
For $I\subseteq\mathbb{N}$ let $L_I$ denote the compositum
of the $(L_i)_{i\in I}$.
By (i), $L/\mathbb{Q}$ is Galois with Galois group
isomorphic to the procyclic group $\prod_{i=1}^\infty\mathbb{Z}/p_i\mathbb{Z}$,
hence every number field $N$ contained in $L$
is of the form
$N=L_{I_N}$ where $I_N=\{i:p_i|[N:\mathbb{Q}]\}$.
For $K\subseteq F\subseteq M$ we get $I_K\subseteq I_F\subseteq I_M$.
Moreover, with $M'=L_{I_M\mysetminus I_F}$
we get that $M=FM'$ and $F\cap M'=\mathbb{Q}$,
in particular $[M:F]=[M':\mathbb{Q}]$.
The map $F\mapsto I_M\mysetminus I_F$ is a bijection between
intermediate fields $K\subseteq F\subseteq M$
and subsets of $I_M\mysetminus I_K$.

By (ii), the discriminants of the number fields $F$ and $M'$ are coprime,
and therefore
\[
 |\Delta_M|=|\Delta_F|^{[M':\Q]}\cdot|\Delta_{M'}|^{[F:\Q]},
\]
cf.~\cite[Thm.~4.26]{Nar}.
Thus
\[
 \gamma_{M}(F)=\frac{\left(|\Delta_F|^{[M':\Q]}\cdot|\Delta_{M'}|^{[F:\Q]}\right)^{\frac{1}{[M:\Q][M:F]}}}{|\Delta_F|^\frac{1}{[M:\Q]}}=|\Delta_{M'}|^{\frac{1}{[M':\Q]^2}}.
 \]
Therefore 
\[
 \gamma'(M/K)=\sup_{K\subseteq F\subseteq M}\gamma_M(F)=\sup_{J\subseteq I_M\mysetminus I_K} |\Delta_{L_J}|^{\frac{1}{[L_J:\Q]^2}}.
\]
Since (again by (ii)), $L_J$ is the compositum of fields with pairwise coprime discriminants, 
\[
 |\Delta_{L_J}|=\prod_{i\in J}|\Delta_{L_i}|^{[L_J:L_i]}=\prod_{i\in J}|\Delta_{L_i}|^{\frac{[L_J:\mathbb{Q}]}{[L_i:\mathbb{Q}]}},
\]
and thus 
\[
|\Delta_{L_J}|^{\frac{1}{[L_J:\Q]^2}}=\prod_{i\in J}\left(|\Delta_{L_i}|^{\frac{1}{[L_i:\Q]^2}}\right)^{\frac{[L_i:\Q]}{[L_J:\Q]}}\leq \left(\sup_{i\in\mathbb{N}}|\Delta_{L_i}|^{\frac{1}{[L_i:\Q]^2}}\right)^{\frac{\sum_{i\in J}[L_i:\Q]}{\prod_{i\in J}[L_i:\Q]}}\leq \sup_{i\in\mathbb{N}}|\Delta_{L_i}|^{{1/p_i^2}}.\qedhere
\]
\end{proof}

Now Theorem \ref{main-th} easily follows from \cite{CF}:

\begin{proof}[Proof of Theorem \ref{main-th}] 
The field $L$ with the Northcott property constructed in \cite[\S3, Proof of Theorem 1.6]{CF} 
is obtained as a compositum of a sequence $(F_i)_{i\in\mathbb{N}}$
of cyclic Galois extensions of $\mathbb{Q}$ with $[F_i:\mathbb{Q}]=p_i$ pairwise distinct prime numbers, $|\Delta_{F_i}|^{1/p_i^2}\leq 3$
(this is property (c) in that proof)
and ${\rm gcd}(\Delta_{F_i},\Delta_{F_1}\cdots\Delta_{F_{i-1}})=1$
(this follows from property (b) in that proof 
and is stated explicitly in the middle of p.~2410).
Therefore, by Proposition \ref{main-details} (with $L_i=F_i$),
$L$ satisfies
\begin{equation}\label{eq:leq3}
 \liminf_{K\subseteq M\subset L}\gamma'({M}/K)\leq 3   
\end{equation}
for every number field $K$ contained in $L$.
As $L/\Q$ is abelian, so is $L/K$.
In particular, every such $M/K$ is Galois, hence by Corollary \ref{simply-Wid} we have that 
\[
 \liminf_{K\subseteq M\subset L}\gamma({M/K})=\liminf_{K\subseteq M\subset L}\gamma'({M/K}).
\]  
In particular, by (\ref{eq:leq3}) we obtain that
$\liminf_{K\subseteq M\subset L}\gamma({M/K})<\infty$,
so $L$ is an infinite Galois extension of $\mathbb{Q}$
that does not satisfy condition $(\ref{eqn:liminf})$ of Theorem \ref{new-crit-Wid}
for any number field $K\subseteq L$.
\end{proof}

\section{PAC fields with the Northcott property revisited}\label{sec-PAC}

\noindent
Amoroso, David and Zannier asked in \cite{ADZ} whether there exist subfields of $\overline{\mathbb{Q}}$ that are pseudo-algebraically closed (PAC) and have the so-called Bogomolov property, which is a weakening of the Northcott property.
Pottmeyer in \cite[Example 3.2]{Pottmeyer} constructs such fields,
and \cite[Prop.~1.2]{Feh} gives a different construction of PAC fields that even have the stronger Northcott property.
The goal of this section is to show
how this construction can be significantly simplified.

\begin{lemma}\label{lem:lindisj}
Let $K$ be a number field
and $(L_i)_{i\in\mathbb{N}}$ a pairwise linearly disjoint family of proper finite extensions of $K$ (within $\overline{\mathbb{Q}}$) of the same degree.
Then
$$
 \liminf_{i\in\mathbb{N}}\inf_{\alpha\in L_i\mysetminus K}h(\alpha) = \infty.
$$
\end{lemma}

\begin{proof}
As $(L_i\mysetminus K)\cap (L_j\mysetminus K)=\emptyset$ for every $i\neq j$ by the linear disjointness assumption,
this is an immediate consequence of Northcott’s theorem  
mentioned in the introduction.
\end{proof}

\begin{proposition}
For every number field $K$ there exists an algebraic extension $L/K$ such that $L$ is PAC and has the Northcott property.
\end{proposition}

\begin{proof}
Let $(C_n)_{n\in\mathbb{N}}$ be an enumeration of the geometrically integral curves over $K$
that admit a Galois morphism $\pi_n\colon C_n\rightarrow\mathbb{P}^1$.
Let $K_0=K$
and iteratively construct $K_{n+1}$ from $K_n$ as follows:
By Hilbert’s irreducibility theorem (e.g. in the form \cite[Lemma 16.2.6]{FJ}), there is an infinite sequence
$L_1,L_2,\dots$ of specializations 
of $K_n(C_n)/K_n(\mathbb{P}^1)$
such that 
$[L_i:K_n]={\rm deg}(\pi_n)$ for every $i$,
and the family $(L_i)_{i\in\mathbb{N}}$ is linearly disjoint over $K_n$.
In particular, it is pairwise linearly disjoint over $K_n$,
so by Lemma \ref{lem:lindisj} there is some $K_{n+1}=L_i$ with 
$\inf_{\alpha\in K_{n+1}\mysetminus K_n}h(\alpha)>n$.

Let $L=\bigcup_{n\in\mathbb{N}}K_n$.
For every $n$, $C_n(L)\supseteq C_n(K_{n+1})\neq\emptyset$.
Since every geometrically integral curve $C$ over $K$ is dominated
by some $C_n$ (this follows from 
the so-called stability of fields theorem \cite[Theorem 18.9.3]{FJ}),
we get that $C(L)\neq\emptyset$
for every such $C$.
Since $L/K$ is algebraic, this implies that $L$ is PAC (\cite[Theorem 11.2.3]{FJ}).
Moreover, for every $n$, Northcott’s theorem gives that 
$$
 \{\alpha\in L:h(\alpha)<n\}\subseteq\{\alpha\in K_n:h(\alpha)<n\}
$$
is finite, so $L$ has the Northcott property.
\end{proof}

The first construction of a PAC field with the Northcott property in \cite{Feh} differs from the short proof given here
in that it needed a more elaborate version of Hilbert's irreducibility theorem that allowed a certain control
on the ramification of the extensions $L_i/K_n$,
so that then Widmer's criterion could be applied to conclude that $L=\bigcup_{n\in\mathbb{N}}K_n$ has the Northcott property.
The short proof here needs only the existence of infinitely many linearly disjoint specializations and can conclude with Northcott's theorem rather than Widmer's criterion.

\section*{Acknowledgements}

\noindent
The authors would like to thank Martin Widmer for interesting discussion.

Sara Checcoli’s work is supported by the French National Research Agency in the framework of the Investissements d’avenir program (ANR-15-IDEX-02). 
Arno Fehm was supported by the Institut Henri Poincar\'e.

\end{document}